\documentclass{amsart}
\usepackage{amsthm}
\usepackage{amsmath}
\usepackage{amssymb}
\usepackage{euscript}
\usepackage{graphicx}
\usepackage{pgf,tikz}
\usetikzlibrary{arrows}
\usepackage[singlelinecheck=false,justification=justified]{caption}

\usepackage{multicol}

\usepackage{pgfplots}
\usetikzlibrary{decorations.pathreplacing}

\pgfplotsset{compat=1.18}

\newtheorem{theorem}{Theorem}[section]
\newtheorem{lemma}[theorem]{Lemma}

\newtheorem{definition}[theorem]{Definition}

\newtheorem{proposition}[theorem]{Proposition}
\newtheorem{example}[theorem]{Example}

\newtheorem{question}[theorem]{Question}

\numberwithin{equation}{section}

\newcommand{\seq}[1]{\langle #1\rangle}
\newcommand{\ns}{\varnothing}

\begin{document}




\title[Neighborhood $N$-Shadowing]{Shadowing with small diameter sets of bounded cardinality}

\author[J. Meddaugh]{Jonathan Meddaugh}
\address[J. Meddaugh]{Department of Mathematics, Baylor University, Waco TX, 76798}
\email[J. Meddaugh]{Jonathan\_Meddaugh@baylor.edu}

\author[E. Stephens]{Elyssa Stephens}
\address[E. Stephens]{Department of Mathematics, Baylor University, Waco TX, 76798}
\email[E. Stephens]{Ellie\_Stephens2@baylor.edu}

\thanks{The first author was supported by a grant from the Simons Foundation (960812, JM)}

\subjclass[2000]{54H20, 37B10, 37C50}
\keywords{neighborhood $N$-shadowing, shadowing, mixing, sofic, specification, subshift of quasi-finite type}
\date{}

\begin{abstract} We examine dynamical systems with the property that pseudo-orbits can be traced by small diameter sets with bounded cardinality. In particular, we show that mixing sofic subshifts and surjective dynamical systems with the specification property have this property, and in these systems, it is sufficient to consider small sets of cardinality no more than two. We also prove that a more general class of subshifts, the \textit{subshifts of quasi-finite type}, exhibit this property. 
	
\end{abstract}
\maketitle

\label{sec:intro}
\section{Introduction}

 For a continuous function $f$ on a compact metric space and $\delta >0$, a $\delta$-pseudo-orbit for the system $(X,f)$ is a sequence $(x_i)_{i\in \omega}$ in $X$ such that $d(f(x_i),x_{i+1})<\delta$ for all $i\in \omega$. Pseudo-orbits are of significant import in modeling dynamical systems---in particular, they are the natural output of finite-precision methods of orbit computation in a system $(X,f)$. Systems in which pseudo-orbits are good approximations of true orbits are, therefore, quite useful \cite{shadowingasastructualproperty}. Such systems are said to have the shadowing property (sometimes known as the pseudo-orbit tracing property). In addition to this import to modeling, shadowing has strong applications in the broader theory of dynamical systems. In particular, one of the earliest definitions of the shadowing property can be accredited to Bowen and his studies on non-wandering sets in Axiom $A$ diffeomorphisms \cite{bowenshadowintro}. 

Since then, many variations of the shadowing property have been introduced and studied. In \cite{orbitalshadowingICT} the authors introduce the \textit{eventual shadowing property} for which pseudo-orbits are shadowed after some length of time. The \textit{unique shadowing property} and its relationship to expansivity is discussed in \cite{expansivityanduniqueshadow}. Other shadowing properties include \textit{limit shadowing} \cite{eirola1997limit}, \textit{$\underline{d}$-shadowing} \cite{subshadowing}, and \textit{thick shadowing} \cite{subshadowing}. In systems with certain dynamical properties, various shadowing properties are also shown to be equivalent. For example, in chain transitive systems \cite{BRIAN_MEDDAUGH_RAINES_2015}. Also, in systems with shadowing, various dynamical properties are known to be equivalent \cite{anoteonaverageshadowing}.

Of particular import to this paper is the notion of using a finite set, rather than a single point, to shadow a pseudo-orbit. This idea was originally introduced in \cite{multishadowing}, called the \emph{multishadowing property}. Systems with multishadowing have the property that pseudo-orbits are shadowed by a finite set of cardinality dependent upon the pseudo-orbit. While many dynamical systems on a compact metric space have the multishadowing property but not the shadowing property (e.g. the constant map on $[0,1]$), the authors in \cite{n-shadowing} introduced a stronger form of shadowing in which the shadowing set has cardinality that is not dependent on the pseudo-orbit. In particular, if pseudo-orbits for a system $(X,f)$ can be shadowed by sets of cardinality no more than $N$, $(X,f)$ is said to have the \emph{$N$-shadowing property}. In this paper, we expand upon this idea and introduce a stronger property, which we call the \emph{neighborhood $N$-shadowing property}, requiring the shadowing set to be small not only in cardinality, but also in diameter. Briefly, a system has \emph{neighborhood $N$-shadowing} if for every $\varepsilon>0$ there exists a $\delta>0$ such that every $\delta$-pseudo-orbit is $\varepsilon$-shadowed by a finite set $A \subseteq X$ with $|A|\leq N$ and $diam(A)<\varepsilon.$

The structure of the paper is as follows. In Section \ref{sec:preliminaries}, we introduce key definitions and concepts, including that of \emph{neighborhood $N$-shadowing}. In Section \ref{sec:NeighborhoodSofic}, we demonstrate that a large class of subshifts exhibit neighborhood $2$-shadowing but not shadowing. In Section \ref{sec:generalnshadow}, we explore connections between neighborhood $N$-shadowing, mixing, and specification. We close with some open questions and discussion in Section \ref{sec:conclusion}.

\section{Preliminaries} \label{sec:preliminaries}

For the purposes of this paper, a \emph{dynamical system} is a pair, $(X,f)$,  where $X$ is a compact metric space and $f:X\to X$ is a continuous map. The \emph{orbit} of a point $x\in X$ is the sequence $(f^n(x))_{n\in \omega}$ in $X$, where $\omega=\mathbb N\cup \{0\}$. A point $x\in X$ is \emph{periodic} if there exists $p\in \mathbb{N}$ such that $f^p(x)=x$. 

Let $(X,f)$ be a dynamical system and $\delta>0$. A \emph{$\delta$-pseudo-orbit} is a sequence $(x_i)_{i\in \omega} $ in $X$ such that $d(f(x_i), x_{i+1})<\delta $ for all $i\in \omega.$ A dynamical system $(X,f)$ has the \emph{shadowing property} provided that for every $\varepsilon>0$ there exists $\delta>0$ such that for every $\delta$-pseudo-orbit, $(x_i)_{i\in \omega}$, there exists a point $z\in X$ such that $d(f^i(z),x_i)<\varepsilon$ for all $i\in \omega.$ In this case, the $\delta$-pseudo-orbit is said to be $\varepsilon$-\emph{shadowed} by $z$.

In systems $(X,f)$ in which $f$ is a homeomorphism, we often wish to discuss analogous bi-directional notions. In particular, in this case, the \emph{two-sided orbit} of $x\in X$ is the bi-infinite sequence $(f^i(x))_{i\in\mathbb{Z}}$, and a \emph{two-sided $\delta$-pseudo-orbit} is a sequence $(x_i)_{i\in\mathbb Z}$ such that $d(f(x_i),x_{i+1})<\delta$ for all $i\in\mathbb Z$. The system $(X,f)$ has the \emph{two-sided shadowing property} provided that for each $\varepsilon>0$, there exists $\delta>0$ such that  for every two-sided $\delta$-pseudo-orbit, $(x_i)_{i\in \mathbb Z}$, there exists a point $z\in X$ such that $d(f^i(z),x_i)<\varepsilon$ for all $i\in \mathbb Z.$ In this case, the two-sided $\delta$-pseudo-orbit is said to be $\varepsilon$-\emph{shadowed} by $z$.

The principal object of study in this paper is a weaker notion of shadowing in which we allow for a small \emph{set} of points to shadow pseudo-orbits. 

\begin{definition}
	Let $N\in\mathbb N$. A dynamical system $(X,f)$ has the \emph{neighborhood $N$-shadowing property} provided that for each $\varepsilon>0$ there exists $\delta >0$ such that for every $\delta$-pseudo-orbit $(x_i)_{i\in\omega}$, there exists $A\subseteq X$ with $|A|\leq N$ and $diam(A)<\varepsilon$ such that $d(f^i(A),x_i) <\varepsilon$ for all $i\in\omega$. In this case, we say that $(x_i)_{i\in\omega}$ is $\varepsilon$-shadowed by $A$.
\end{definition}

It is worth pointing out the following elementary observations: $1$-shadowing is equivalent to the standard notion of shadowing and that if $N\leq M$, then neighborhood $N$-shadowing implies neighborhood $M$-shadowing.

The notion of neighborhood $N$-shadowing is inspired by the $N$-shadowing property, which was introduced in \cite{n-shadowing}. Fix $N\in \mathbb{N}$. A dynamical system $(X,f)$ has the \emph{$N$-shadowing property} if for every $\varepsilon>0$, there exists $\delta>0$ such that for every $\delta$-pseudo-orbit $(x_i)_{i\in \omega}$, there exits a set $A\subseteq X$ with $|A|\leq N$ such that $d(f^i(A), x_i)<\varepsilon$ for all $i\in\omega.$ 

It is immediately clear that neighborhood $N$-shadowing implies $N$-shadowing. The following example (which appears in \cite{n-shadowing}) demonstrates that the converse is false, i.e. neighborhood $N$-shadowing is strictly stronger than $N$-shadowing.

\begin{example} Let $f:[0,1]\to[0,1]$ be the function indicated below.
	\begin{multicols}{2}
		\begin{minipage}{.5\textwidth}
			$$f(x)= \begin{cases}
				\displaystyle\frac{1}{2}\sqrt{2x} & 0\leq x <\displaystyle \frac{1}{2}\\\\
				\displaystyle \frac{1}{2}+ \displaystyle\frac{1}{2}\sqrt{2x-1} & \displaystyle \frac{1}{2} \leq x \leq 1 
			\end{cases}$$
		\end{minipage}\columnbreak
		\begin{minipage}{.5\textwidth}
			
			\begin{tikzpicture}[scale=.6]
				\begin{axis}[
					axis lines=middle,
					axis line style={-},  
					enlargelimits=false,  
					domain=0:1,
					samples=100,
					xtick={1},
					ytick={1},
					xticklabel={1},
					yticklabel={1},
					xlabel={\empty},
					ylabel={\empty},
					yticklabel style={xshift=-.2cm}, 
					xticklabel style={yshift=-.2cm}, 
					tick style={major tick length=0pt, minor tick length=0pt},
					grid=both,
					grid style={draw=black},
					axis on top
					]
					\addplot[domain=0:.5,thick] {(1/2)*sqrt(2*x)};
					\addplot[domain=.5:1,thick] {0.5+(1/2)*sqrt(2*x-1)};
					\addplot[domain=0:1] {x};
					
					\node at (axis cs:.2,.5) {$f$};
					\node at (axis cs:.25,.1) {$y=x$};
				\end{axis}
			\end{tikzpicture}
			
		\end{minipage}
	\end{multicols}
	It was shown in \cite{n-shadowing} that $f$ has 2-shadowing, but not shadowing. To see that it also does not have neighborhood $N$-shadowing for any $N$, fix $\varepsilon=\frac{1}{4}$. It is easy to check that for any $\delta>0$, there are $\delta$-pseudo-orbits $(x_i)_{i\in\omega}$ with $x_0=0$ and $x_K=1$ for sufficiently large $K$. However, if $diam(A)<\varepsilon$ and $A$ contains a point within $\varepsilon$ of $0$, then $A\subseteq [0,2\varepsilon)\subseteq [0,1/2)$ and $d(1,f^i(A))>1/2$ for all $i\in\omega$, so $A$ does not $\varepsilon$-shadow any such pseudo-orbit.
\end{example}

A dynamical system $(X,f)$ has the \emph{specification property} if for every $\varepsilon >0$ there exists a positive integer $M>0$ such that for any $s\geq 2$, any finite collection of points $\{y_1,y_2,...,y_s\}\subseteq X$ and sequence of positive integers $0\leq j_1\leq k_1<j_2\leq k_2<\cdot\cdot\cdot < j_s\leq k_s$ with $j_{m+1}-k_m \geq M $ for all $1\leq m \leq s-1, $ there exists a point $x\in X$ such that $d(f^i(x),f^i(y_m))< \varepsilon$ for all $j_m\leq i\leq k_m$ and $1\leq m\leq s$. It is worth pointing out that some authors require the point $x$ to have period $M+k_s$, but we make no such restriction in this paper. 

A dynamical system $(X,f)$ is \emph{transitive} if for any pair of nonempty open sets $U,V\subseteq X$, there exists $n\in \mathbb{N}$ such that $f^n(U)\cap V\neq \emptyset.$ The system is \emph{weakly mixing} if for all nonempty open sets $U_1,U_2,V_1,V_2\subseteq X$, there exists $n\in \mathbb{N}$ such that $f^n(U_1)\cap V_1\neq \emptyset$ and $f^n(U_2)\cap V_2\neq \emptyset$. The system is \emph{mixing} if for every pair of nonempty  open subsets $U,V\subseteq X$, there exists $N\in \mathbb{N}$ such that for all $n\geq N$, $f^n(U)\cap V\neq \emptyset.$ It is well-known that specification is strictly stronger than mixing, which is strictly stronger than transitivity \cite{vries2014topological}.

 For a finite set $\mathcal{A}$, consider the set $\mathcal A^\mathbb Z$ of bi-infinite sequences in $\mathcal A$. For $x=(x_i)_{i\in\mathbb Z}\in\mathcal A^\mathbb Z$ and integers $i<j$, we define $x_{[i,j)}=x_ix_{i+1}\ldots x_{j-1}$ and $x_{[i,j]}=x_{[i,j+1)}$. If $w=x_{[i,j)}$ for some $i<j$, then we say that $w$ is a \emph{subword} of $x$. If $u$ and $v$ are subwords of elements $\mathcal A^\mathbb Z$, then we understand $uv$ to be the concatenation of the two words. 
 
 The \emph{full shift with alphabet $\mathcal A$} is the dynamical system $(\mathcal{A}^\mathbb{Z},\sigma)$ where $\sigma$ is the \emph{shift map} defined by $\sigma: \mathcal{A}^\mathbb{Z}\rightarrow \mathcal{A}^\mathbb{Z}$ defined by $(\sigma(x))_i = x_{i+1}$ and $\mathcal A^\mathbb Z$ is endowed with the metric
 \[d(x,y) = \inf\{2^{-i}: x_{[-i,i]}\neq y_{[-i,i]}\}.\]
 Since there is no ambiguity in the map under discussion, we will often refer $\mathcal A^\mathbb Z$ as the full shift, omitting mention of $\sigma$. The \emph{forward shift with alphabet $\mathcal A$} is defined analogously, with  $\mathcal A^\omega$ replacing $\mathcal A^\mathbb Z$. All notation/terminology/results below apply equally well to forward shifts.

  A \emph{subshift of $\mathcal A^\mathbb Z$} is a shift-invariant, closed subsystem $(X,\sigma)$ of the full shift $\mathcal A^\mathbb Z$.  For a subshift $X$ (again, as there is no ambiguity in the map, we will sometimes use refer to a subshift by its domain, omitting $\sigma$) and $n\in\mathbb N$, we define $\mathcal B_n(X)=\{x_{[i,i+n)}:x\in X\}$ to be the set of \emph{allowable words of length $n$ in $X$} and $\mathcal B(X)=\bigcup_{n\in\mathbb N}B_n(X)$ to be the \emph{language} of $X$. 
  For each subshift $X$ of $\mathcal A^\mathbb Z$, there exists a set $\mathcal F$ of \emph{forbidden words} such that $X=\{x\in\mathcal A^\mathbb Z: \textrm{for all } i<j, x_{[i,j)}\notin \mathcal F\}$. If the subshift $X$ has a finite set of forbidden words that defines it, it is called a \emph{subshift of finite type}. Subshifts of finite type are characterized as exactly those subshifts with the shadowing property \cite{Walters1978}.

\emph{Sofic subshifts} are subshifts that can be realized using a labeled graph in the following sense. A subshift $X$ of $\mathcal A^\mathbb Z$ is \emph{sofic} provided that there exists a directed graph $G$ (with edge set $\mathcal{E}$ and vertex set $\mathcal{V}$) and a \emph{labeling} $\mathcal L:\mathcal E\to\mathcal A$, such that $X$ is the set consisting of the images under $\mathcal L$ of the bi-infinite (edge) walks in $G$.

We close this section by stating the following well-known result concerning mixing sofic subshifts.

 \begin{proposition} \label{soficmixingnumber}
 	A sofic subshift $X$ is mixing if and only if there exists $M\in\mathbb N$ such that for all $u,v\in\mathcal B(X)$ and $n\geq M$, there exists $w\in\mathcal B_n(X)$ such that $uwv\in \mathcal B(X)$.
 \end{proposition}

For any subshift $X$, we will refer to an $M$ satisfying the conditions of this proposition as a \emph{mixing number} for $X$. It is clear that a subshift with a mixing number is mixing, but the converse fails in general.

\section{Neighborhood $N$-Shadowing in Subshifts} \label{sec:NeighborhoodSofic}
We now focus our attention on neighborhood $N$-shadowing in subshifts. As subshifts have much more structure than general dynamical systems, it is not uncommon for nonequivalent dynamical properties to be equivalent in some classes of subshifts. In particular, it is natural to ask if neighborhood $N$-shadowing might be equivalent to  $N$-shadowing or shadowing in this setting. We open with the following example which demonstrates that this is not the case.

\begin{example} \label{ex:neighborhoodVsnon}
	
	The sofic subshift $X$, presented by the labeled graph $\mathcal{G}$ below 
	 has $2$-shadowing, but not neighborhood $N$-shadowing for any $N\in\mathbb N$. 
	
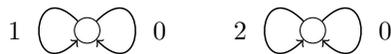
		\begin{center}
		\begin{tikzpicture}
			[auto=left,every node/.style={circle}]
			\node[draw = black] (1) at (-5,0){};
			\path[->,line width=.2mm] (1) [out=135,in=225,looseness=10] edge node[left] {$1$} (1);  
			\path[->,line width=.2mm] (1) [out=45,in=315,looseness=10] edge node[right] {$0$} (1);  
			\node[draw = black] (2) at (-2,0){};
			\path[->,line width=.2mm] (2) [out=135,in=225,looseness=10] edge node[left] {$2$} (2);  
			\path[->,line width=.2mm] (2) [out=45,in=315,looseness=10] edge node[right] {$0$} (2);  
		\end{tikzpicture}
			\end{center}
That $X$ does not have neighborhood $N$-shadowing (for any $N\in\mathbb N$) follows by taking $\varepsilon=1/2$ and noting that for any $\delta>0$, there exists $n$ sufficiently large so that $\langle x^i\rangle_{i\in\omega}$ given by

\[
	x^0 = 1\underbrace{00\cdots00}_{n}\cdots\quad
	x^1= \underbrace{00\cdots00}_{n}2\cdots  \quad \text{and }
	x^i=\sigma^{i-1}(x^1) \text{ for } i>1
\]
		is a $\delta$-pseudo-orbit. If $A$ is a subset of $X$ of diameter less than $\varepsilon$ containing at least one point within $\varepsilon$ of $x^0$, then each point in $A$ contains a $1$, and therefore none of its shifts contain a $2$. Thus, for all $x\in A$, $d(\sigma^{n+1}(x),x^{n+1})>\varepsilon$, so $A$ does not $\varepsilon$-shadow this pseudo-orbit. 
		
		To see that $X$ has $2$-shadowing, fix $1>\varepsilon>0$ and take $\delta=\varepsilon$. 
		If $\langle x^i\rangle_{i\in\omega}$ is a $\delta$-pseudo-orbit, define $c=(c_i)_{i\in\omega}$ by taking $c_i=x^i_0$ for each $i\in\omega$. 
		Now, define $a=(a_i)_{i\in\omega}$ by taking $a_i=0$ if $c_i=0$ and $a_i=1$ otherwise. Similarly, define $b=(b_i)_{i\in\omega}$ by taking $b_i=0$ if $c_i=0$ and $b_i=2$ otherwise. It is immediately clear that both $a$ and $b$ belong to $X$ and it is not difficult to check that for each $n\in\omega$, we have at least one of $d(\sigma^n(a),x^n)<\varepsilon$ or $d(\sigma^n(b),x^n)<\varepsilon$.
	
\end{example}

It is well known that the subshifts with shadowing are precisely the subshifts of finite type \cite{Walters1978}. Example \ref{ex:neighborhoodVsnon} demonstrates that there are subshifts with $2$-shadowing which are not subshifts of finite type, so it is natural to ask whether there are subshifts with neighborhood $2$-shadowing which are not subshifts of finite type. The following results show that the class of subshifts with neighborhood $2$-shadowing is indeed much larger than the subshifts of finite type.

\begin{theorem}\label{twosidedmixing}
	Let $(X, \sigma)$ be a subshift with a mixing number. Then $(X, \sigma)$ has two-sided neighborhood $2$-shadowing. In particular, every mixing sofic subshift has two-sided neighborhood $2$-shadowing.
\end{theorem}

\begin{proof}
	
	Let $(X, \sigma)$ be a subshift with a mixing number. Fix $\varepsilon>0$ and let $M$ be a mixing number for $(X,\sigma)$. Choose $N>M$ so that if $a,b\in X$ with $a_{(-N,N)}=b_{(-N,N)}$, then $d(a,b)<\varepsilon$. Also, choose $\delta>0$ such that if $a,b\in X$ with $d(a,b)<\delta$, then $a_{[-4N,4N]}=b_{[-4N,4N]}$.
	
	Now, let $\seq{x^i}_{i\in\mathbb Z}$ be a two-sided $\delta$-pseudo-orbit, and define the $c=(c_i)_{i\in\mathbb Z}\in\mathcal A^\mathbb Z$ by taking $c_i=x^i_0$. By our choice of $\delta$, it is easy to check that for a fixed $i\in\mathbb Z$ and any $j\in[-4N,4N]$, we have $x^i_j=x^{i+j}_0=c_{i+j}$, from which it follows that $x^i_{[-4N,4N]}=c_{[i-4N,i+4N]}$. In particular, since $x^i\in X$ for each $i\in\mathbb Z$, we have $c_{[i-4N,i+4N]}\in\mathcal B(X)$ for all $i\in\mathbb Z$. It follows that for each $t\in\mathbb Z$, we also have $c_{[t,t+5N]}\in\mathcal B(X)$.
	
	We now construct points $a,b\in X$ with the property that $a_{(-N,N)}=b_{(-N,N)}=x^0_{(-N,N)}$ (and hence $d(a,b)$, $d(a,x)$ and $d(b,x)$ are all less than $\varepsilon$) and for each $i\in\mathbb Z$, we have that \(x^i_{(-N,N)}\in\{a_{(i-N,i+N)}, b_{(i-N,i+N)}\}\).
	Since $a_{(i-N,i+N)}=(\sigma^i(a))_{(-N,N)}$ and $b_{(i-N,i+N)}=(\sigma^i(b))_{(-N,N)}$, it follows that one of $d(x^i,\sigma^i(a))$ or $(x^i,\sigma^i(b))$ is less than $\varepsilon.$ In other words, the set $\{a,b\}$ has diameter less than $\varepsilon$ and $\varepsilon$-shadows $\seq{x_i}_{i\in\mathbb Z}$. 
	
	We now construct the point $a$. First, for $j\in\mathbb N$, define $L_j=(5-6j)N$ and $R_i=(6j-2)N$. Note that for each $j\in\mathbb N$, $c_{[N+R_j,R_{j+1})}$ and $c_{(L_{j+1},L_j-N]}$ both belong to $\mathcal B(X)$ since $R_{j+1}-N-R_j=L_j-N-L_{j+1}=5N$. Begin by taking $a_{(L_1,R_1)}=c_{(L_1,R_1)}$.Now, assuming that $a_{(L_j,R_j)}$ has been defined and belongs to $\mathcal B(X)$, we define $a_{[R_j,R_{j+1})}$ and $a_{(L_{j+1},L_j]}$ as follows. Since $a_{(L_j,R_j)}$ and $c_{[N+R_j,R_{j+1})}$ belong to $\mathcal B(X)$ and $N>M$, we can find $w_j\in\mathcal B_N(X)$ with $a_{(L_j,R_j)}w_jc_{[N+R_j,R_{j+1})}\in\mathcal B(X)$. We can then find $u_j\in\mathcal B_N(X)$ with 
	\[a_{(L_{j+1},R_{j+1})}=c_{(L_{j+1},L_j-N]}u_ja_{(L_j,R_j)}w_jc_{[N+R_j,R_{j+1})}\in\mathcal B(X).\]

	This process defines a point $a\in X$. Note that since $L_j\leq-N<N\leq R_j$, we have $a_{(-N,N)}=c_{(-N,N)}$ as desired. In addition, we have that $a_{[N+R_j,R_{j+1})}=c_{[N+R_j,R_{j+1})}$ and $a_{(L_{j+1},L_j-N]}=c_{(L_{j+1},L_j-N]}$ for all $j\in\mathbb N$.
		
	We define the point $b$ similarly. We begin by defining, for $j\in\mathbb N$, $S_j=(2-6j)N$ and $T_j=(6j-5)N$ and taking $b_{(S_1,T_1)}=c_{(S_1,T_1)}$. Using the fact that $N$ is a mixing number for $X$, we proceed as above, defining $b_{(S_j,T_j)}$ recursively. The resulting point $b$ has the properties that $b_{(-N,N)}=c_{(-N,N)}$ and that $b_{[N+T_j,T_{j+1})}=c_{[N+T_j,T_{j+1})}$ and $b_{(S_{j+1},S_j-N]}=c_{(S_{j+1},S_j-N]}$ for all $j\in\mathbb N$.
	
	All that remains is to verify that \(x^i_{(-N,N)}\in\{a_{(i-N,i+N)}, b_{(i-N,i+N)}\}\) for each $i\in\mathbb Z$. Towards this end, fix $i\in\mathbb Z$ and suppose that $x^i_{(-N,N)}\neq b_{(i-N,i+N)}$. Since $x^i_{(-N,N)}=c_{(i-N,i+N)}$, it follows that $(i-N,i+N)$ is not a subset of  $(S_{j+1},S_j-N]$ or of $[N+T_j,T_{j+1})$ for any $j\in\mathbb N$ (as $c$ and $b$ agree in those intervals). However \[\mathbb Z\setminus\bigcup_{j\in\mathbb N}(S_{j+1},S_j-N]\cup [N+T_j,T_{j+1})\subseteq\bigcup_{j\in\mathbb Z}[(6j+1)N,(6j+2)N],\] and so we can find $j\in\mathbb Z$ with $(i-N,i+N)\cap[(6j+1)N,(6j+2)N]\neq\ns$. It follows that $(i-N,i+N)\subseteq((6j-1)N,(6j+4)N)$. We can also find $k\in\mathbb N$ with $((6j-1)N,(6j+4)N)$ a subset of either $[N+R_k,R_{k+1})$ or $(L_{k+1},L_k-N]$, the intervals on which $a$ and $c$ agree. It follows that
	\(x^i_{(-N,N)}=c_{(i-N,i+N)}=a_{(i-N,i+N)}\).

	\end{proof}

Unsurprisingly, an analogous result (with analogous proof) holds for forward subshifts.

\begin{theorem}\label{Prop:SoficForwardNSHadow}
	Let $(X,\sigma)$ be a forward subshift with a mixing number. Then $(X,\sigma)$ has neighborhood $2$-shadowing. In particular if $(X,\sigma)$ is a  mixing sofic forward subshift, then it has neighborhood $2$-shadowing.
\end{theorem}

While all mixing sofic subshifts have neighborhood 2-shadowing, the following is an example of a subshift without mixing but which still has neighborhood 2-shadowing.

\begin{example}\label{ex:NotMixing} The sofic subshift presented by the labeled graph $\mathcal{H}$, below, is not mixing but has neighborhood 2-shadowing.

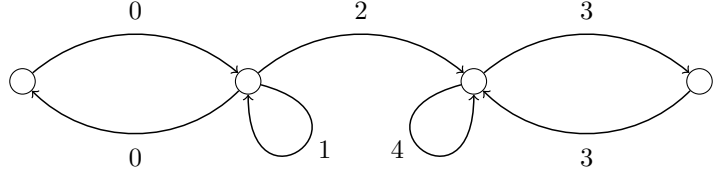
\begin{center}
    		\begin{tikzpicture}
			[auto=left,every node/.style={circle}]
			\node[draw = black] (1) at (-5,0){};
			\node[draw = black] (2) at (-2,0){};
			\node[draw = black] (3) at (1,0){};
			\node[draw = black] (4) at (4,0){};
			
			\path[->,line width=.2mm] (2) [out=-15,in=-90,looseness=20] edge node[right] {$1$} (2);
			\path[->,line width=.2mm] (2) [out=225,in=-45,looseness=1] edge node[below] {$0$} (1);
			\path[->,line width=.2mm] (1) [out=40,in=140,looseness=1] edge node[above] {$0$} (2);    
			\path[->,line width=.2mm] (2) [out=40,in=140,looseness=1] edge node[above] {$2$} (3);    
			\path[->,line width=.2mm] (3) [out=195,in=270,looseness=20] edge node[left] {$4$} (3);
			\path[->,line width=.2mm] (4) [out=225,in=-45,looseness=1] edge node[below] {$3$} (3);
			\path[->,line width=.2mm] (3) [out=40,in=140,looseness=1] edge node[above] {$3$} (4);    
		\end{tikzpicture}
\end{center}
That this subshift is not mixing is seen by noting that no point beginning with a $3$ can have a $0$ or $1$ occuring later in the sequence. The proof of Theorem \ref{twosidedmixing} can be adapted (noting that if $\langle x^i\rangle$ is a $\delta$-pseudo-orbit for $\delta<1$ and $x^i$ contains a $2$, $3$, or $4$, then for $j>i$, $x^j$  cannot have $0$ or $1$ appearing anywhere) to prove that this subshift has neighborhood $2$-shadowing. This can also be proven by applying Theorem \ref{thm:semifinitetype}.
\end{example}

In light of this example, we define a generalization of the finite type condition that is sufficient to ensure neighborhood $2$-shadowing in subshifts. Recall that subshifts of finite type have the following alternative characterizations \cite{Lind_Marcus_1995}; a subshift $(X,f)$ is a subshift of finite type if and only if there exists $M\in\mathbb N$ such that for all $u,v\in\mathcal B(X)$, if $w\in\mathcal B_n(X)$ with $uv,vw\in\mathcal B(X)$ and $n\geq M$, then $uvw\in\mathcal B(X)$.

\begin{definition}	
	A shift space $(X,\sigma)$ is a subshift of quasi-finite type if there exists $M\in \mathbb{N}$ with the property that for all $u,v\in \mathcal{B}(X)$, if there exists $w\in \mathcal{B}_n(X)$ with $uw,wv\in \mathcal{B}(X)$ and $n\geq M$, then there exists $z\in \mathcal{B}_n(X)$ such that $uzv\in \mathcal{B}(X)$. In this case, we call $M\in \mathbb{N}$ a \textit{subshift of quasi-finite type number.} 
\end{definition}

It is clear that each subshift of finite type is trivially a subshift of quasi-finite type, as are subshifts with mixing numbers. That there are subshifts of quasi-finite type which are neither mixing nor of finite type follows from the following observation.

\begin{example}
	The sofic subshift $(X,\sigma)$ of Example \ref{ex:NotMixing} is neither mixing nor of finite type, but is a subshift of quasi-finite type. That $(X,\sigma)$ is not mixing follows from the discussion in \ref{ex:NotMixing}. That it is not of finite type follows by observing that for all $n\in\mathbb N$, we have $10^n,0^n1\in\mathcal B(X)$, but if $n$ is odd, $10^n1\notin\mathcal B(X)$.
	
	To show that $(X,\sigma)$ is a subshift of quasi-finite type, let $Y$ be the even shift and $M\geq 5$ be a mixing number for $Y$. We show that $M$ is a subshift of quasi-finite type number for $X.$ Note that for any $u,v\in \mathcal{B}(Y)$ and $n\ge M$ it is easy to verify there exists $z\in \{ 1^n, \; 01^{n-2}0, \; 1^{n-1}0,\; 01^{n-1} \}\subset \mathcal{B}_n(Y)$ such that $uzv\in \mathcal{B}(Y)$.

To show that $M$ is a subshift of quasi-finite type number for $X$, define $\pi: X\to Y $ for $x\in X$ by replacing any 2 or 4 in $x$ with a 1 and any 3 in $x$ with a 0. Suppose that $u,v,w\in \mathcal{B}(X)$ with $|w|=n\ge M$ so that $uw,\; wv\in \mathcal{B}(X)$. Let $\pi(u)=u'$ and $\pi(v)=v'.$ Since $M$ is a mixing number for $Y$ and $n\ge M$, by our above argument, we may find $z'\in \mathcal{B}_n(Y)$ such that $u'z'v'\in \mathcal{B}(Y)$ and $z'_j=1$ for some $j\in \omega$. 

If $v$ contains a 2, 1, or 0, then $w$ contains only 1's and 0's. Then $z'\in \pi^{-1}(\{z'\})$ and $uz'v\in \mathcal{B}(X)$. If $u$ contains a 2, 3, or 4, then $w$ contains only 3's and 4's. Then we may find $z\in \pi^{-1}(\{z'\} )$ such that $$z_i=\begin{cases}
    3\; &z_i'=0\\
    4\; &z_i'=1
\end{cases}$$
and $uzv\in \mathcal{B}(X)$. If $w$ contains a 2, find $z\in \pi^{-1}(\{z'\} )$ such that $$z_i=\begin{cases}
    2 \; &i=j\\
    0\; &w_i'=0,\; i<j\\
    1\; &w_i'=1,\; i<j\\
    3\; &w_i'=0,\; i>j\\
    4\; &w_i'=1,\; i>j\\
\end{cases}$$ and $uzv\in \mathcal{B}(X).$ Therefore, $M$ is a subshift of quasi-finite type number for $X$.

\end{example}

We now prove the following generalization of Theorem \ref{twosidedmixing}. The proof proceeds in a nearly identical manner--the only significant difference is in the details of the construction of the points $a$ and $b$ where mixing was applied in the earlier proof.

\begin{theorem} \label{thm:semifinitetype}
	
	Let $(X,\sigma)$ be a subshift of quasi-finite type. Then $(X,\sigma)$ has neighborhood $2$-shadowing.
	
\end{theorem}

\begin{proof}
	Let $(X, \sigma)$ be a subshift of quasi-finite type. Fix $\varepsilon>0$ and let $M$ witness that $(X,\sigma)$ is a subshift of quasi-finite type. Choose $N>M$ so that if $a,b\in X$ with $a_{(-N,N)}=b_{(-N,N)}$, then $d(a,b)<\varepsilon$. Also, choose $\delta>0$ such that if $a,b\in X$ with $d(a,b)<\delta$, then $a_{[-4N,4N]}=b_{[-4N,4N]}$.
	
	Now, let $\seq{x^i}_{i\in\mathbb Z}$ be a two-sided $\delta$-pseudo-orbit, and define the $c=(c_i)_{i\in\mathbb Z}\in\mathcal A^\mathbb Z$ by taking $c_i=x^i_0$. By our choice of $\delta$, it is easy to check that for a fixed $i\in\mathbb Z$ and any $j\in[-4N,4N]$, we have $x^i_j=x^{i+j}_0=c_{i+j}$, from which it follows that $x^i_{[-4N,4N]}=c_{[i-4N,i+4N]}$. In particular, since $x^i\in X$ for each $i\in\mathbb Z$, we have $c_{[i-4N,i+4N]}\in\mathcal B(X)$ for all $i\in\mathbb Z$. It follows that for each $t\in\mathbb Z$ we have $c_{[t,t+7N]}\in\mathcal B(X)$.
	
	We now construct points $a,b\in X$ with the property that $a_{(-N,N)}=b_{(-N,N)}=x^0_{(-N,N)}$ (and hence $d(a,b)$, $d(a,x)$ and $d(b,x)$ are all less than $\varepsilon$) and for each $i\in\mathbb Z$, we have that \(x^i_{(-N,N)}\in\{a_{(i-N,i+N)}, b_{(i-N,i+N)}\}\).
	Since $a_{(i-N,i+N)}=(\sigma^i(a))_{(-N,N)}$ and $b_{(i-N,i+N)}=(\sigma^i(b))_{(-N,N)}$, it follows that one of $d(x^i,\sigma^i(a))$ or $(x^i,\sigma^i(b))$ is less than $\varepsilon.$ In other words, the set $\{a,b\}$ has diameter less than $\varepsilon$ and $\varepsilon$-shadows $\seq{x_i}_{i\in\mathbb Z}$. 
	
	We now construct the point $a$. First, for $j\in\mathbb N$, define $L_j=(5-6j)N$ and $R_i=(6j-2)N$. Note that for each $j\in\mathbb N$, 
		\[c_{[R_j,N+R_{j+1}]}=c_{[R_j,N+R_j]}c_{(N+R_j,R_{j+1})}c_{[R_{j+1},N+R_{j+1}]}\in\mathcal B(X)\]
	and
		\[c_{[L_{j+1}-N,L_j]}=c_{[L_{j+1}-N,L_{j+1}]}c_{(L_{j+1},L_j-N)}c_{[L_{j}-N,L_j]}\in\mathcal B(X),\] 
	since $N+R_{j+1}-R_j=L_j-(L_{j+1}-N)=7N$. 
	
	Begin by taking $a_{(L_1,R_1)}=c_{(L_1,R_1)}$ and observe that
	  
		\[c_{[L_1-N,L_1]}a_{(L_1,R_1)}c_{[{R_1,N+R_1}]}=c_{[L_1-N,N+R_1]}\in\mathcal B(X)\] 
	since $R_1+N-(L_1-N)=7N$. Now, assuming that $a_{(L_j,R_j)}$ has been defined so that 
		\[c_{[L_j-N,L_j]}a_{(L_j,R_j)}c_{[R_j,N+R_j]}\in\mathcal B(X),\] 
	we define $a_{(N+R_j,N+R_{j+1}]}$ and $a_{[L_{j+1}-N,L_j-N)}$ as follows. Since $c_{[L_j-N,L_j]}a_{(L_j,R_j)}c_{[R_j,N+R_j]}$ and $c_{[R_j,N+R_j]}c_{(N+R_j,R_{j+1})}c_{[R_{j+1},N+R_{j+1}]}$ belong to $\mathcal{B}(X)$ and $c_{[R_j,N+R_j]}\in\mathcal B_{N+1}(X)$ with $N+1>M$, we can find and $w_j\in\mathcal B_{N+1}(X)$  with \[c_{[L_j-N,L_j]}a_{(L_j,R_j)}w_jc_{(N+R_j,R_{j+1})}c_{[R_{j+1},N+R_{j+1}]}\in\mathcal B(X).\] Similarly, we can then find $u_j\in\mathcal B_{N+1}(X)$ with 
\[c_{[L_{j+1}-N,L_j]}c_{(L_{j+1},L_j-N)}u_ja_{(L_j,R_j)}w_jc_{(N+R_j,R_{j+1})}c_{[R_{j+1},N+R_{j+1}]}\in\mathcal B(X).\]
	We now take $a_{(L_{j+1},R_{j+1})}=c_{(L_{j+1},L_j-N)}u_ja_{(L_j,R_j)}w_jc_{(N+R_j,R_{j+1})}$.
	
	This process defines a point $a\in X$. Note that since $L_j\leq-N<N\leq R_j$, we have $a_{(-N,N)}=c_{(-N,N)}$ as desired. In addition, we have that $a_{[N+R_j,R_{j+1})}=c_{[N+R_j,R_{j+1})}$ and $a_{(L_{j+1},L_j-N]}=c_{(L_{j+1},L_j-N]}$ for all $j\in\mathbb N$.
	
	We define the point $b$ similarly. We begin by defining, for $j\in\mathbb N$, $S_j=(2-6j)N$ and $T_j=(6j-5)N$ and taking $b_{(S_1,T_1)}=c_{(S_1,T_1)}$. Using the fact that $N>M$, we proceed as above, defining $b_{(S_j,T_j)}$ recursively. The resulting point $b$ has the properties that $b_{(-N,N)}=c_{(-N,N)}$ and that $b_{[N+T_j,T_{j+1})}=c_{[N+T_j,T_{j+1})}$ and $b_{(S_{j+1},S_j-N]}=c_{(S_{j+1},S_j-N]}$ for all $j\in\mathbb N$.
	
	All that remains is to verify that \(x^i_{(-N,N)}\in\{a_{(i-N,i+N)}, b_{(i-N,i+N)}\}\) for each $i\in\mathbb Z$. Towards this end, fix $i\in\mathbb Z$ and suppose that $x^i_{(-N,N)}\neq b_{(i-N,i+N)}$. Since $x^i_{(-N,N)}=c_{(i-N,i+N)}$, it follows that $(i-N,i+N)$ is not a subset of  $(S_{j+1},S_j-N]$ or of $[N+T_j,T_{j+1})$ for any $j\in\mathbb N$ (as $c$ and $b$ agree in those intervals). However \[\mathbb Z\setminus\bigcup_{j\in\mathbb N}(S_{j+1},S_j-N]\cup [N+T_j,T_{j+1})\subseteq\bigcup_{j\in\mathbb Z}[(6j+1)N,(6j+2)N],\] and so we can find $j\in\mathbb Z$ with $(i-N,i+N)\cap[(6j+1)N,(6j+2)N]\neq\ns$. It follows that $(i-N,i+N)\subseteq((6j-1)N,(6j+4)N)$. We can also find $k\in\mathbb N$ with $((6j-1)N,(6j+4)N)$ a subset of either $[N+R_k,R_{k+1})$ or $(L_{k+1},L_k-N]$, the intervals on which $a$ and $c$ agree. It follows that
	\(x^i_{(-N,N)}=c_{(i-N,i+N)}=a_{(i-N,i+N)}\).
	
\end{proof}

\section{Neighborhood $N$-shadowing in Dynamical Systems} \label{sec:generalnshadow}

In this section, we examine the neighborhood $N$-shadowing property in general dynamical systems.  In particular, we begin by demonstrating an analogue of Theorem \ref{twosidedmixing} and Corollary \ref{Prop:SoficForwardNSHadow} for surjective dynamical systems, using the stronger property of specification. It is well-known that in surjective systems, specification implies mixing, whereas the converse is false in general. Interestingly, in sofic subshifts and in surjective systems with shadowing, the two properties are equivalent (\cite{mixequivtospec} and \cite{anoteonaverageshadowing}, respectively).

We begin with the following straightforward result allowing the specification property to be applied to an infinite set of points $\{y_i:u\in\omega\}$ rather than only finitely many points. It is worth pointing out that in non-compact settings, this infinite form of specification is strictly stronger than the standard specification, see \cite{hammon2025specification}.

\begin{lemma}\label{speclemma}
	
	Let $(X,f)$ be a dynamical system on a compact metric space having the specification property. Then for every $\eta>0$ there exists $M$ such that for any sequence $(y_n)_{n\in \omega}$ in $X$ and any sequence $0\leq j_0\leq k_0<j_1\leq k_1<\dots < j_s\leq k_s\dots$ with $j_{s+1}-k_s \geq M $ for all $s$, there exists $z\in X$ such that for $s\in \omega$, $d(f^i(z),f^i(y_s))< \eta$ for all $j_s\leq i\leq k_s$. 
\end{lemma}

\begin{proof}
	
	Let $\eta>0$. Find $M$ such that the specification property holds for $\eta/2.$ Let $(y_n)_{n\in \omega}$ be a sequence in $ X $. Fix sequences $\langle j_i\rangle$ and $\langle k_i\rangle$ such that $0\leq j_0\leq k_0<j_1\leq k_1<\dots < j_s\leq k_s\dots$ with $j_{m+1}-k_m \geq M $ for all $1\leq m$.
	
	By the specification property, for each $m\in \omega$, find $z_m\in X$ such that for $0\leq s\leq m$ and $j_s\leq i\leq k_s$, we have $d(f^i(z_m),f^i(y_s))<\eta/2$.
	
	Since $X$ is compact, we may assume without loss that the sequence $\langle z_m\rangle_{m\in\omega}$ converges to a point $z\in X$. To complete the proof, fix $s$ and $i$ with $j_s\leq i\leq k_s$. Then $d(f^i(z),f^i(y_s)=\lim_{m\to\infty}d(f^i(z_m),f^i(y_s)\leq \eta/2<\eta$.
\end{proof}

We now prove the main result of the section.

\begin{theorem}
	
	Let $X$ be a compact metric space and $f:X\to X$ be a continuous surjection with the specification property. Then $(X,f)$ has neighborhood 2-shadowing.
	
\end{theorem}

\begin{proof}
	
	Let $\varepsilon>0$. Fix $M$ satisfying the conclusion of Lemma \ref{speclemma}, taking $\eta=\varepsilon/2.$ By compactness of $X$ and continuity of $f$ and its iterates, find $0<\delta<\varepsilon/4$ such that if $d(a,b)<\delta,$ then $d(f^i(a),f^i(b))<\varepsilon/8M$ for all $0\leq i\leq 2M$. 
	
	Now, let $(x_i)_{i\in \omega} $ be a $\delta$-pseudo-orbit. Notice that for $i\in\omega$ and $n\leq 2M$,
	\begin{align*}
		d(f^n(x_i),x_{i+n})&\leq \sum_{j=1}^n d\left(f^{1+n-j}(x_{i+{j-1}}), f^{n-j}(x_{i+j})\right)\\
		&= \sum_{j=1}^n d\left(f^{n-j}(f(x_{i+{j-1}})), f^{n-j}(x_{i+j})\right)\\
		&< \frac{(n-1)\varepsilon}{8M}+\delta<\frac{2M\varepsilon}{8M}+\frac{\varepsilon}{4}<\frac{\varepsilon}{2}
	\end{align*}

	Now consider the sequences $\langle j_s\rangle_{s\in \omega}$ and $\langle k_s\rangle_{s\in\omega}$ defined by $j_s = 2sM$ and $k_s= j_s+M$. Note that these sequences satisfy the hypotheses of Lemma \ref{speclemma} relative to choice of $M$ and taking $\eta=\varepsilon/2$.
	
	By surjectivity of $f$, for each $s\in\omega$, fix $y_s\in X$ such that $f^{j_s}(y_s)=x_{j_s}$.
	
	By applying Lemma \ref{speclemma}, find $z\in X$ such that for each $s$ and $i$ in $\omega$ with $j_s\leq i \leq k_s$, we have $d(f^i(z),f^i(y_s))<\varepsilon/2$. Note that since $j_s\leq i \leq k_s$, we have $f^i(y_s)=f^{i-j_s}(x_{j_s})$. Thus, for $j_s\leq i\leq k_s$, we have (since $i-j_s\leq 2M$)
	\begin{align*}
		d(f^i(z),x_i)&\leq d(f^i(z),f^{i-j_s}(x_{j_s}))+d(f^{i-j_s}(x_{j_s}),x_i)\\
		&=	d(f^i(z),f^{i}(y_{s}))+d(f^{i-j_s}(x_{j_s}),x_{(j_s)+(i-j_s)})<\varepsilon. 
	\end{align*}

	In particular, for $i\in\bigcup[j_s,k_s]$, we have that  $d(f^i(z),x_i)<\varepsilon$.
	
	Now consider the sequences $\langle p_s\rangle_{s\in \omega}$ and $\langle q_s\rangle_{s\in\omega}$ defined by taking $p_0=0$, $p_s=k_s$ for $s>0$, and $q_s=j_{s+1}$. Using an argument for these sequences identical to the preceding, choose $w\in X$ so that for $i\in\bigcup[p_s,q_s]$ we have $d(f^i(w),x_i)<\varepsilon$.
	
	Taking $A=\{z,w\}$, we see that $\seq{x_i}$ is $\varepsilon$-shadowed by $A$ since $(\bigcup[j_s,k_s])\cup(\bigcup[p_s,q_s])=\omega$. Furthermore, since $p_0=j_0=0$, we have $d(z,w)\leq d(z,x_0)+d(w,x_0)>\varepsilon/2+\varepsilon/2$ so that the diameter of $A$ is less than $\varepsilon$.

\end{proof}

It is also the case that this generalized form of shadowing can be used to derive other dynamical properties in a system. We say that a system $(X,f)$ has \emph{dense small periodic sets} provided that for all $U\subseteq X$ open, there exists $A\subseteq U$ compact and $n>0$ with $f^n(A)=A$. The following result is in the spirit of the results from \cite{anoteonaverageshadowing}. 

\begin{proposition}\label{weakmixingdensesmall}
	Let $(X,f)$ be a continuous weakly mixing map with dense small periodic sets such that $f$ satisfies the neighborhood $N$-shadowing property for some $N\in\mathbb N$. Then $(X,f)$ is mixing. 
\end{proposition}

\begin{proof}
	Let $U,V\subseteq X$ be nonempty open subsets of $X.$ Find $\varepsilon>0$, $u\in U$, and $v\in V$ such that $B_{3\varepsilon}(u) \subseteq U $ and $B_{3\varepsilon}(v) \subseteq V $.
	
	Find $\delta>0$ such that $d(x,y)<\delta$ implies $d(f(x), f(y))<\varepsilon/2.$ Without loss of generality, assume $\delta<\varepsilon/2.$
	
	By the neighborhood $N$-shadowing property, we may find $0<\gamma<\delta/2$ such that for every $\gamma$-pseudo-orbit, $(x_i)_{i\in \omega}$, there exists a subset $K\subseteq X$ with $|K|\leq N$ and $diam(K)<\delta/2$ such that for all $i\in \omega,$ $d(f^i(K), x_i)<\delta/2.$
	
	Since $f$ has dense small periodic sets, choose $A\subseteq B_{\gamma/4}(u) $ such that $f^n(A)=A$ for some $n\in \mathbb{N}.$ Note $B_{\gamma/4}(u)\subseteq B_{\delta/2}(u) \subseteq B_{\varepsilon/2}(u).$
	
	Fix $a\in A$ and let $(a_k)_{k\in \omega} $ be defined by $a_k = f^{k\; (\text{mod } n)}(a). $ Note that $(a_k)_{k\in \omega}$ is a $\gamma$ pseudo-orbit. 
	
	By Proposition 1.6.2 in \cite{vries2014topological}, find $m\in \mathbb{N}$  such that for each $j\in \{0,...,n-1\}$, we may find $z_{j}\in B_{\gamma/4}(u)$ such that $f^{m+j}(z_j) \in B_{\varepsilon}(v) $. Let $l\geq m$ and notice that $l=m+j_0+sn$ for some $s\geq0$ and $j_0\in \{0,...,n-1\}.$ Define $(y_t)_{t\in \omega} $ by $$y_t = \begin{cases}
		a_t, \; t\leq sn-1\\
		f^{t-sn}(z_{j_0}), \; t\geq sn
	\end{cases}$$
	
	Note that $f(y_{sn-1}) = a_{sn} \in A. $ Since $a_{sn},z_{j_0}\in B_{\gamma/4}(u) $, then $d(a_{sn},z_{j_0})<\gamma$, so $(y_t)_{t\in \omega}$ is a $\gamma $-pseudo-orbit. 
	
	Now find $K\subseteq X$ such that $|K|\leq N,$ $diamK<\delta/2,$ and for all $t\in \omega,$ $d( y_t, f^t(K))<\delta/2.$
	
	Fix $b\in K$ such that $d(b,a)<\delta/2.$ Then for any $b'\in K,$ $$d(b',a)\leq d(b',b)+ d(a, b)< \delta/2+ \delta/2 =\delta.$$
	Hence, $K \subseteq B_{\delta}(a)\subseteq B_{2\delta}(u) \subseteq B_{\varepsilon}(u)\subseteq U. $
	
	Note that $y_l = f^{m+j}(z_{j_0})\in B_{\varepsilon}(v) $. Since $d(f^l(K), y_l)<\delta/2,$ find $x\in K\subseteq U$ such that $d(f^l(x), f^{m+j_0}(z_{j_0}))< \delta/2. $ Then $f^l(x) \in B_{\delta/2}(f^{m+j_0}(z_{j_0})) \subseteq B_{2\varepsilon}(v)  \subseteq V. $
	
	Since $l\geq m$ was arbitrary, we have for all $l\geq m,$ there exists $x\in U$ such that $f^l(x) \in V.$ Hence, $f$ is mixing.

\end{proof}

Note that in the results of \cite{anoteonaverageshadowing}, shadowing and weakly mixing together are enough to yield the required dense collection of small periodic sets. For the more general notion of shadowing we use, it is not clear whether the same is true.

\section{Closing Remarks} \label{sec:conclusion}

In Section \ref{sec:NeighborhoodSofic}, we demonstrated that there is a large class of subshifts which exhibit neighborhood $2$-shadowing, but not shadowing (i.e. neighborhood $1$-shadowing). Strikingly, while systems with $(N+1)$-shadowing, but not $N$-shadowing are known to exist \cite{n-shadowing} (and indeed, Example \ref{ex:neighborhoodVsnon} easily generalizes to show the same in subshifts), we have not been able to find a system or subshift which exhibits neighborhood $3$-shadowing, but not neighborhood $2$-shadowing. This leads us to ask the following questions.

\begin{question}
	For which $N$ are there dynamical systems with neighborhood $(N+1)$-shadowing, but not neighborhood $N$-shadowing?
\end{question}

\begin{question}
	For which $N$ are there subshifts with neighborhood $(N+1)$-shadowing, but not neighborhood $N$-shadowing?
\end{question}

Also in Section \ref{sec:NeighborhoodSofic}, we introduced the notion of a subshift of quasi-finite type and demonstrated that every subshift with this property has neighborhood $2$-shadowing. While it would be quite satisfying if the two properties were equivalent, this does not seem to be the case. In particular, in our proof of Theorem \ref{thm:semifinitetype}, the two points which shadow the pseudo-orbit both shadow the pseudo-orbit in a fairly uniform fashion. In particular, the set $\{i\in\mathbb Z:f(\sigma^i(a),x^i)<\varepsilon\}$ is a syndetic set (i.e. there is an upper bound on the difference of consecutive members), which is quite a bit stronger than is required of $2$-shadowing. 

\begin{question}
	Are there subshifts with neighborhood $2$-shadowing which are not a subshift of quasi-finite type?
\end{question}

An affirmative answer to this question naturally leads one to ask the following fairly open-ended questions.

\begin{question}
	Can being a subshift of quasi-finite type be reasonably characterized as a form of shadowing in subshifts?
\end{question}

\begin{question}
	Can neighborhood $2$-shadowing in subshifts be reasonably characterized by the language of a subshift?
\end{question}

Finally, in Section \ref{sec:generalnshadow}, we proved a partial generalization of a result of Kwietniak and Oprocha. In their paper \cite{anoteonaverageshadowing}, the hypothesis of the existence of dense small periodic sets is unnecessary as it follows from the weak mixing combined with the shadowing property. The neighborhood $N$-shadowing property seems as though it may be sufficient to yield the same result, but this remains an open question.

\begin{question}
	Is there a weakly mixing system which has the neighborhood $N$-shadowing property ($N>1$) that does not have dense small periodic sets?
	\end{question}

\bibliographystyle{plain}
\bibliography{NeighborhoodShadowingBib}
\end{document}